\documentclass[12pt]{article}
\usepackage{amsmath,amssymb,amsthm, amsfonts}
\usepackage{hyperref}
\usepackage{graphicx, enumerate}
\usepackage{url}
\usepackage[mathlines]{lineno}
\usepackage{dsfont} 
\usepackage{color}
\definecolor{red}{rgb}{1,0,0}

\definecolor{blue}{rgb}{0,0,1}

\definecolor{green}{rgb}{0,.6,0}

\usepackage{float}

\usepackage{tikz}

\newtheorem{thm}{Theorem}[section]
\newtheorem{cor}[thm]{Corollary}
\newtheorem{lem}[thm]{Lemma}
\newtheorem{prop}[thm]{Proposition}

\newtheorem{obs}[thm]{Observation}

\newcommand{\throt}{\operatorname{th}}
\newcommand{\throtplus}{\operatorname{th_+}}
\newcommand{\thp}[1]{\throtplus(#1)} 
\newcommand{\pt}{\operatorname{pt}}
\newcommand{\ptp}{\operatorname{pt_+}}
\newcommand{\ptpm}[2]{\ptp(#1 ; #2)}
\theoremstyle{definition}
\newtheorem{rem}[thm]{Remark}

\theoremstyle{definition}
\newtheorem{defn}[thm]{Definition}

\theoremstyle{definition}
\newtheorem{ex}[thm]{Example}

\newcommand{\G}{\mathcal{G}}
\newcommand{\F}{\mathcal{F}}

\newcommand{\Z}{\operatorname{Z}}

\newcommand{\ecc}{\operatorname{ecc}}

\newcommand{\diam}[1]{\operatorname{diam}\left(#1\right)}
\newcommand{\dist}{\operatorname{dist}}

\newcommand{\bit}{\begin{itemize}}
\newcommand{\eit}{\end{itemize}}
\newcommand{\ben}{\begin{enumerate}}
\newcommand{\een}{\end{enumerate}}
\newcommand{\beq}{\begin{equation}}
\newcommand{\eeq}{\end{equation}}
\newcommand{\bea}{\begin{eqnarray*}} 
\newcommand{\eea}{\end{eqnarray*}}
\newcommand{\bpf}{\begin{proof}}
\newcommand{\epf}{\end{proof}\ms}
\newcommand{\bmt}{\begin{bmatrix}}
\newcommand{\emt}{\end{bmatrix}}
\newcommand{\ms}{\medskip}

\newcommand{\cp}{\, \Box\,}
\newcommand{\lc}{\left\lceil}
\newcommand{\rc}{\right\rceil}
\newcommand{\lf}{\left\lfloor}
\newcommand{\rf}{\right\rfloor}

\newcommand{\noi}{\noindent}
\newcommand{\ceil}[1]{\lc #1 \rc}
\newcommand{\OL}{\overline}

\title{Throttling positive semidefinite zero forcing propagation time on graphs}
\author{Joshua Carlson\thanks{Department of Mathematics, Iowa State University, Ames, IA 50011, USA, 
(jmsdg7, hogben, jkritsch, lorenkj, msross, sselken, vvalle)@iastate.edu.} \and Leslie Hogben\footnotemark[1]\ \thanks{American Institute of Mathematics, 600 E. Brokaw Road, San Jose, CA 95112, USA, hogben@aimath.org.} \and J\"urgen Kritschgau\footnotemark[1] \and Kate Lorenzen\footnotemark[1] \and Michael S. Ross\footnotemark[1] \and Seth Selken\footnotemark[1] \and Vicente Valle Martinez\footnotemark[1]}

\date{June 1, 2018}

\begin{document}
\maketitle

\vspace{-25pt}\begin{abstract} 
Zero forcing is a process on a graph  that colors vertices blue by starting with some of the vertices blue and applying a color change rule.
Throttling minimizes the sum of the size of the initial  blue vertex set and the number of the time steps needed to color the graph.  We study throttling for positive semidefinite zero forcing.
We establish a tight lower bound on the positive semidefinite throttling number as a function of the order, maximum degree, and positive semidefinite zero forcing number of the graph, and 
determine the positive semidefinite throttling numbers of paths, cycles, and full binary trees.  We characterize the graphs that have extreme  positive semidefinite throttling numbers.  \end{abstract}

\noi {\bf Keywords} Zero forcing, propagation time, throttling, positive semidefinite

\noi{\bf AMS subject classification} 05C57, 05C15, 05C50

\section{Introduction}\label{sintro}

Consider a process on a graph  wherein the vertices are colored either blue or white, and we repeatedly apply a color change rule that can change the color of a white vertex to blue but not vice versa. Natural questions arise such as the final state of the graph after this process and the time needed for this process to end. 
Butler and Young \cite{BY13throttling} studied  the relationship between the size of the initial set colored blue and the number of  time steps taken to color the entire graph. Motivating applications include studying the spread of information on a graph \cite{BY13throttling}, graph searching \cite{yang2012fast}, and control of quantum systems \cite{graphinfect, Sev}.

Throughout this paper we consider only simple (no loops or multiple edges) undirected finite graphs $G=(V(G),E(G))$.
The \textit{standard color change rule} consists of changing the color of a white vertex $w$ to blue  when $w$ is the only white neighbor of a blue  vertex $v$.  We then say that $v$ forces $w$. 
A subset $S$ of the vertices initially colored blue that can eventually force all vertices of $G$ 
is called a \textit{standard zero forcing set}. 
The minimum cardinality of a standard zero forcing set for $G$ is the {\em standard zero forcing number} of $G$ and is denoted by $\Z(G)$ \cite{AIM08}.  
The number of time steps required for this process to color all vertices blue (performing all possible forces at each  step) is the \textit{standard propagation time} of a set $S$, denoted by $\pt(G,S)$ \cite{proptime}. 
Whenever $S$ is not a standard zero forcing set we let $\pt(G,S)=\infty$.
Note that each vertex forces at most one of its neighbors under the standard color change rule.
Starting at a blue vertex $v_1$, a sequence of forces $v_i\to v_{i+1}$ for $i=1,\dots,k-1$ implies the path $(v_1,\dots, v_k)$ is an induced path in $G$. Such a path is called a \textit{forcing chain}; whenever we refer to a forcing chain, we assume it is maximal. 

We study positive semidefinite  (PSD) throttling using positive semidefinite zero forcing, introduced in \cite{smallparam}.
Let $W_1,...,W_k$ be the sets of white vertices corresponding to the connected components of $G-S$, where $S$ is a set of blue vertices (it is possible that $k=1$). 
The \textit{positive semidefinite color change rule} consists of coloring $w_i\in W_i$ blue when $w_i$ is the only white neighbor of $v$ in $G[W_i \cup S]$, the subgraph of G induced by $W_i \cup S$. We then say that $v$ forces $w$ and write $v\rightarrow w$.  
A subset $S$ of the vertices initially colored blue that can eventually force all vertices of $G$ under the positive semidefinite color change rule is called a \textit{positive semidefinite zero forcing set}. 
The minimum cardinality of a positive semidefinite zero forcing set for $G$ is the {\em positive semidefinite zero forcing number} of $G$ and is denoted by $\Z_+(G)$.  
The number of time steps required for this process to color all vertices is the \textit{positive semidefinite propagation time} of set $S$, denoted by $\ptp(G;S)$;  if $S$ is not a positive semidefinite zero forcing set, then $\ptp(G;S)=\infty$.   The \textit{positive semidefinite propagation time} of graph $G$ is $\ptp(G)=\min\{\ptp(G;S)\!:\! G\mbox{ is a minimum PDS zero forcing set of }G\}$ \cite{PSDpropTime}.  
Given a graph $G$,  a positive semidefinite zero forcing set $S$, a list of forces $\F$, and a vertex  $x\in S$, define $V_x$ to be the set of vertices  $w$ such that  there is a sequence of forces  $x=v_1\to v_2\to\dots\to v_k=w$ in $\F$ (the empty sequence of forces is permitted, i.e., $x\in V_x$).
The {\em forcing tree $T_x$}  is the induced subgraph $T_x=G[V_x]$.  Note that for a given positive semidefinite zero forcing set $S$, there are usually choices to be made in selecting $\F$, and these choices affect the forcing tree $T_x$.  Whenever we refer to a forcing tree, we assume it is maximal.

 In \cite{BY13throttling} Butler and Young  define $\throt(G,S)=|S|+\pt(G,S)$ for  $S\subseteq V(G)$. The \textit{throttling number} of $G$ is $\throt(G)=\min \{ \throt(G,S)\!:\! S\text{ is a zero forcing set}\}$. 
At the AIM workshop {\em Zero forcing and its applications} \cite{AIM-ZF}, one of the problems posed was to study throttling numbers of variants of standard zero forcing.  We address this question for positive definite zero forcing by defining $\throtplus(G;S)=|S| + \ptp(G;S)$ and the \textit{positive semidefinite throttling number} of a graph $G$   as  
\[ \throtplus(G) = \text{min} \left\{ \throtplus(G;S)\!  :  S \text{ is a PSD zero forcing set}    \right\}\!.
\]

We develop positive semidefinite analogs of many of the results in \cite{BY13throttling}, although in many case the results are strikingly different.  
  In Section \ref{bounds} we obtain a lower bound on the positive semidefinite throttling number. Unlike the case of standard throttling, where $\throt(G)\ge 2\sqrt{n} - 1$  for a graph $G$ of order $n$ \cite{BY13throttling}, the maximum degree plays a critical role in the lower bound  for positive semidefinite throttling.  Positive semidefinite throttling on a graph having maximum degree two behaves  like standard throttling (with the lower bound smaller by a factor of $\sqrt 2$, see Proposition \ref{d2bound}), whereas for a graph having maximum degree at least three the lower bound is logarithmic in the order  (see Theorem \ref{thmDeltaGreater2}). The positive semidefinite throttling numbers of paths and cycles are determined in Section \ref{pathsAndCycles}.   Section \ref{trees} contains results on positive semidefinite throttling numbers of trees, and we present a family of trees that shows the lower bound in Theorem \ref{thmDeltaGreater2} is tight.  In Section \ref{extreme} we characterize the graphs that have extreme (very low and very high) positive semidefinite throttling numbers.  In Section \ref{sweight} we discuss weighted positive semidefinite throttling, where a linear combination of $|S|$ and $\ptp(G;S)$ is minimized.  In some cases where we obtain results for which the standard throttling analog has not been done, we establish the standard throttling analogs in Section \ref{StndTh}.  This includes the determination of throttling numbers of cycles and some results for trees. 
The remainder of this introduction contains additional definitions and notation.

For $W\subseteq V(G)$, define the {\em complement}  $\overline W = V(G)\setminus W.$   Vertices $v$ and $u$ are {\em adjacent} (or are {\em neighbors}) if $\{u,v\}\in E(G)$; this relationship can be denoted by $v\sim u$. The {\em neighborhood} of $u$ is $N_G(u)=\{v\in V(G)\!:\! v\sim u\}$; when the graph $G$ is clear we write $N(v)$.    A set $W$ of vertices of $G$ is {\em  independent} if no vertex in $W$ is  adjacent to any other vertex of $W.$ The size of the largest independent set is called the \textit{independence number} of $G$ and is denoted by $\alpha(G)$. The maximum and minimum degree of vertices in $G$ are denoted by $\Delta(G)$ and $\delta(G)$, respectively.  
A path (respectively, cycle, complete graph, complete bipartite graph) of order $n$ is denoted by $P_n$ (respectively, $C_n, K_n, K_{p,q}$).  The {\em length} of a path is the number of edges in the path.

The {\em distance} between vertices $u$ and $v$, denoted by $\dist(u,v)$, is the length of the shortest path between $u$ and $v$. The {\em distance from a set $U\subseteq V(G)$ to a vertex $v$} is defined as $\dist(U,v)=\min\{\dist(u,v)\!:\! u \in U\}$. For $U, W \subseteq V(G)$, the {\em distance from $U$ to $W$}  is  $\dist(U\to W)=\max\{ \dist(U,w) \!:\! w\in W\}$.  The definition of the distance from $U$ to $W$ is for use  with zero forcing and follows \cite{PSDpropTime}, although we have changed the notation to emphasize that in general $\dist(U\to W)\ne \dist(W\to U)$. The {\em eccentricity} of vertex $u$ is $\ecc(u)=\max\{\dist(u,v)\!:\! v\in V(G)\}$.  The {\em diameter} of a graph $G$ is $\diam G=\max\{\ecc(v) \!:\! v\in V(G)\}$.   The center of a graph is the set of all vertices $u$ such that $\ecc(u)\le \ecc(v)$ for all $v\in V(G)$.   
We say a vertex is  a center vertex if it belongs to the center of the graph.

Given a starting set of blue vertices $S=S_0$,  $S_t$ denotes the set of blue vertices after time step $t$, and   $S^{(t)}$ denotes the set of vertices that turn blue at time step $t$. 

\section{Positive semidefinite throttling bounds}\label{bounds}

In \cite{BY13throttling}, Butler and Young constructed a zero forcing set $S$ on a  path $P_n$ such that $\pt(P_n,S) + |S|= \lc2\sqrt{n} - 1\rc$.  They showed $\throt(P_n)=\lc2\sqrt{n} - 1\rc$  by using the lower bound \vspace{-3pt}
\[
\throt(G) \geq 2\sqrt{n} - 1  \vspace{-3pt}
\]
for all graphs $G$ of order $n$.  The lower bound was obtained by minimizing $\pt(G,S) + |S|$ subject to the constraint 
\beq|S|\cdot (\pt(G,S) + 1) \geq n,\label{stdconst}\eeq which follows from the facts  that there are $|S|$ forcing chains and at each time step at most one force can take place in each forcing chain. 

We develop a related lower bound. 
However,  \eqref{stdconst} does not hold for positive semidefinite throttling, because for positive semidefinite zero forcing we have forcing trees rather than forcing chains, and forcing trees can grow by more than one vertex in one time step. For example, consider the star on four vertices, $K_{1,3}$, and let $v$ be  the center. Then $S:=\{v\}$ is a positive semidefinite zero forcing set with propagation time equal to one. So 
$|S|\cdot (\ptp(K_{1,3};S) + 1) = 1\cdot(1+1) = 2 < 4 = n.$

\begin{obs}
For every graph $G$, $\throtplus(G) \leq \throt(G)$. 
\end{obs}

\begin{obs}
If $\Delta(G)=0$, then the only positive semidefinite zero forcing set is $V(G)$, so $\throtplus(G)=n$.
\end{obs}


\begin{prop}
Suppose $G$ is a  graph of order $ n$ with $k$  isolated vertices and $\Delta(G) = 1$. Then\vspace{-3pt} 
\bea\throtplus(G) = \frac{n-k}{2} + k + 1.  \eea 
\end{prop}

\bpf 
For any set $S \subseteq V(G)$, necessarily $\pt_+(G;S) \in \{0, 1, \infty\}$. If $\pt_+(G;S) = 0$, then $|S| = n$ and $|S| + \pt_+(G;S) = n$. If $\pt_+(G;S) = 1$, then $|S| \geq \frac{n - k}{2}  + k$. Define $\hat S$ to be the set of all isolated vertices together with one vertex from each copy of $K_2$. Then $|\hat S| + \pt_+(G;\hat S) = \frac{n-k}{2} + k + 1$. 
\epf

\begin{lem}\label{constraint}
Suppose $G$ is a  graph of order $n$ and $S$ is a positive semidefinite zero forcing set of $G$.  Then \vspace{-3pt}
\bea  
n \leq \left\{ \begin{array}{ll}|S|\left(1 + 2 \ptp(G;S)\right)  & \text{if } \Delta(G) = 2\\[.2mm]
|S|\left(1 + \frac{\Delta(G)(\Delta(G) - 1)^{\ptp(G;S)} - \Delta(G)}{\Delta(G) - 2}\right) & \text{if }\Delta(G) > 2
\end{array}\right. .\vspace{-3pt}
\eea 
\end{lem}

\bpf 
We start with $|S|$ vertices colored blue. To determine the maximum possible number of blue vertices after $t$ time steps, we assume a vertex that turns blue at time $t-1$ forces all its neighbors at time $t$.  Of course, forcing will not generally proceed in this manner, but this count produces a valid upper bound on the number of blue vertices after $t$ time steps.   At the first time step, the positive semidefinite color change rule allows at most $|S|  \Delta(G)$ additional vertices to become blue, so $|S^{(1)}| \leq |S| \Delta(G)$. 
For  $2 \leq t \leq \ptp(G;S)$, we have $|S^{(t)}|\le |S^{(t-1)}|(\Delta(G) - 1)$. Thus,  $|S^{(t)}| \leq |S^{(1)}|(\Delta(G) - 1)^{t-1} \leq |S|\Delta(G)(\Delta(G) - 1)^{t-1}$ for  $2 \leq t \leq \ptp(G;S)$. Since $S$ is a positive semidefinite zero forcing set, 
\bea
n &=& |S| + |S^{(1)}| + \displaystyle \sum_{t=2}^{\ptp(G;S)} |S^{(t)}| \\
&\leq& |S| +  |S|\Delta(G)  + \displaystyle\sum_{t=2}^{\ptp(G;S)} |S|\Delta(G)(\Delta(G) - 1)^{t-1}\\[.2mm] 
 &=& |S| + \displaystyle \sum_{t=1}^{\ptp(G;S)}|S|\Delta(G)(\Delta(G) -1)^{t-1}\\ 
&=& |S| \cdot \left[1 + \displaystyle \Delta(G) \sum_{t=0}^{\ptp(G;S) - 1}(\Delta(G) - 1)^t \right]\\[.2mm] 
&=&  \left\{ \begin{array}{ll}|S|(1 + 2 \ptp(G;S))  & \text{if } \Delta(G) = 2\\[.2mm]
|S|\left(1 + \frac{\Delta(G)\left((\Delta(G) - 1)^{\ptp(G;S)} - 1\right)}{(\Delta(G) - 1)-1}\right) & \text{if }\Delta(G) > 2
\end{array}\right. .  \qedhere 
\eea
\epf

\begin{prop}\label{d2bound}
Let $\Delta(G) = 2$.  Then \bea\throtplus(G) \geq \left \lceil \sqrt{2n}-\frac{1}{2} \right \rceil.\eea
\end{prop}

\bpf
For  a positive semidefinite zero forcing set $S$, let $s=|S|$ and $p=\ptp(G;S)$.  We want to minimize the value of $s + p$ subject to $n \leq s(1+2p)$, by Lemma \ref{constraint}. If we allow $s$ and $p$ to be nonnegative real numbers, then for a fixed $s$ we have $p \geq \frac 1 2 \left(\frac n s -1\right)$. Thus, the minimum value of $s+p$  is 
$s+\frac 1 2 \left(\frac n s -1\right)$, achieved by using $p(s):= \frac 1 2 \left(\frac n s -1\right)$ as the value for $p.$  For each $s \geq 0$ define $g(s) = s+\frac 1 2 \left(\frac n s -1\right)$. Then
$g'(s) = 1 - \frac {n} {2s^2}.$
Thus, $g'(s) = 0$ and $s\ge 0$ imply 
$s=\sqrt{\frac{n}{2}},$
and therefore,
\[s+p = \sqrt{\frac{n}{2}}+ \frac 1 2 \left(n\sqrt{\frac{2}{n}} -1\right)=\sqrt{2n}-\frac 1 2.\] 
Since $\sqrt{2n}-\frac 1 2$ is the minimum value of $s+p$ for $0 \leq s,p \in \mathbb{R}$ and $n \leq s(1+2p)$, adding the constraint that $s$ and $p$ are integers with $p \geq 0$ and $s \geq 1$ still gives the bound $g(s) \geq  \sqrt{2n}-\frac 1 2 $, and thus, $g(s) \geq  \lc \sqrt{2n}-\frac 1 2 \rc$.
\epf 
The bound in Proposition \ref{d2bound} is tight. This is easy to verify for some small order paths and cycles and is proved more generally in Theorems \ref{thmThrP} and \ref{thmThrC} for all paths and for cycles of order at least four.

\begin{thm}\label{thmDeltaGreater2}
Let $G$ be a  graph of order $ n$ with $\Delta(G) \geq 3$. Then
\bea 
\throtplus(G) \geq \left\lceil 1 + \log_{(\Delta(G) - 1)}\left(\frac{(\Delta(G)-2)n + 2}{\Delta(G)}\right)\right\rceil.
\eea 
Furthermore, if $\Z_+(G) = s_0$, then
\bea 
\throtplus(G) \geq \left\lceil s_0 + \log_{(\Delta(G) - 1)}\left(\frac{(\Delta(G)-2)n + 2s_0}{\Delta(G)s_0}\right)\right\rceil.
\eea 
\end{thm}

\bpf 
Let $s$, $p$, and $\Delta$ denote $|S|$, $\ptp(G;S)$, and $\Delta(G)$ respectively. By Lemma \ref{constraint},


\begin{eqnarray}
n \leq s\left(1 + \frac{\Delta(\Delta - 1)^p - \Delta}{\Delta - 2}\right)  &\Leftrightarrow& \frac{n}{s} - 1 \leq \frac{\Delta((\Delta - 1)^p - 1)}{\Delta - 2}\nonumber \\[.2mm] 
&\Leftrightarrow& \frac{(\Delta - 2)}{\Delta} \left(\frac{n}{s} - 1\right) + 1 \leq (\Delta - 1)^p \nonumber \\[.2mm]
&\Leftrightarrow& \log_{(\Delta -1)}\left(  \frac{(\Delta - 2)}{\Delta} \left(\frac{n}{s} - 1\right) + 1  \right)  \leq p \nonumber \\[.2mm]
&\Leftrightarrow& \frac{\ln\left(  \frac{(\Delta - 2)n + 2s}{\Delta s}\right)}{\ln(\Delta - 1)} \leq p. \label{pconstraint}
\end{eqnarray}

Define $\displaystyle p(s) =\frac{\ln\left(  \frac{(\Delta - 2)n + 2s}{\Delta s}\right)}{\ln(\Delta - 1)}$. For fixed $s \geq 1$, the minimum  of $s + p$ subject to \eqref{pconstraint} is $s + p(s)$. Note that 
\begin{eqnarray} 
\frac{d}{ds}(s + p(s)) &=& \frac{d}{ds}\left[s + \frac{\ln\left(  \frac{(\Delta - 2)n + 2s}{\Delta s}\right)}{\ln(\Delta - 1)} \right] \nonumber \\[.2mm]
&=& 1 + \frac{1}{\ln(\Delta - 1)}  \frac{d}{ds} \left[ \ln \left(\frac{(\Delta - 2)n}{\Delta} \cdot \frac{1}{s} + \frac{2}{\Delta} \right) \right]  \nonumber   \\[.2mm]
&=& 1 + \frac{1}{\ln(\Delta - 1)}\left(\frac{-(\Delta - 2)n}{\Delta s^2}\right)\left( \frac{\Delta s}{(\Delta - 2)n + 2s} \right) \nonumber \\[.2mm]
&=& 1 - \frac{1}{\ln(\Delta - 1)}\left( \frac{(\Delta -2)n}{s(\Delta - 2)n + 2s^2} \right).  \label{dertiv}
\end{eqnarray}  
Since $s \geq 1$, we can observe that $\frac{d}{ds}(s + p(s)) >0$ when $\Delta \geq 4$ by \eqref{dertiv}. 
So if $\Delta \geq 4$, then $s + p(s)$ is an increasing function, and the minimum occurs at the lowest possible value of  $s$.  

Now suppose $\Delta = 3$ and $s \geq 2$.
Then \eqref{dertiv} becomes 
\beq 
\frac{d}{ds}(s+p(s))= 1 - \frac{1}{\ln 2} \left( \frac{n}{sn+2s^2} \right). \nonumber
\eeq
For all $n \geq 1$,
\bea 
\frac{n}{n+4} < \frac{4}{3} &\Rightarrow& \frac{n}{2n + 8} < \frac{2}{3} < \ln 2\\[.2mm]
&\Rightarrow& \frac{n}{sn + 2s^2} < \ln 2 \\[.2mm]
&\Rightarrow& \frac{1}{\ln 2}\left(\frac{n}{sn + 2s^2}\right) < 1 \\[.2mm]
&\Rightarrow& 0 < 1 -  \frac{1}{\ln 2}\left(\frac{n}{sn + 2s^2}\right)\\[.2mm]
&\Rightarrow& 0 < \frac{d}{ds}(s + p(s)).
\eea 
So when $\Delta = 3$, $s + p(s)$ is increasing for all $s \geq 2$. Thus, $\underset{s\geq 1}{\text{min}} \{s+p(s)\}$ occurs at $s=1$ or $s=2$. Note that
\bea 
\text{min}\{1 + p(1), 2 + p(2) \} &=& \text{min}\left\{1 + \log_2 \left(\frac{n + 2}{3}\right),2 + \log_2 \left(\frac{n + 4}{6}\right)  \right\}\\[.2mm]
&=& \text{min}\left\{\log_2(2) + \log_2 \left(\frac{n + 2}{3}\right),\log_2(4) + \log_2 \left(\frac{n + 4}{6}\right)  \right\}\\[.2mm]
&=& \text{min}\left\{\log_2 \left(\frac{2n + 4}{3}\right),\log_2 \left(\frac{4n + 16}{6}\right)  \right\}\\[.2mm]
&=& \log_2 \left(\frac{2n + 4}{3}\right) = 1 + p(1).
\eea 
This means that for $\Delta \geq 3$, the minimum value of $s + p(s)$ occurs at the minimum value of $s$.
Thus, the minimum value of $s+p$ subject to \eqref{pconstraint} is 
\bea 
1 + p(1) = 1 + \log_{(\Delta(G) - 1)}\left(\frac{(\Delta(G)-2)n + 2}{\Delta(G)}\right)\!.
\eea 

Observe that $s_0=\Z_+(G) \leq |S|$ for any positive semidefinite zero forcing set $S$. So when this value is known, the minimum value 
of $s+p$ subject to $\eqref{pconstraint}$ and $s \ge  s_0$ occurs when $s = s_0$.
\epf

We show that the bound in Theorem \ref{thmDeltaGreater2} is tight by constructing an infinite family of trees that attain the bound in Proposition \ref{TDHBound}. 
The lower bound for $\throtplus(G)$ in Theorem  \ref{thmDeltaGreater2} is attained only by choosing  a minimum zero forcing set.  However, that does not imply that $\throtplus(G)$ can be attained  by choosing a minimum zero forcing set, as seen in Example \ref{ex:thm25notbest}.  First we give a lower bound on positive definite propagation time.

\begin{rem}\label{LBdistcomp}
As noted in \cite{PSDpropTime}, for any $S \subseteq V(G)$, $\pt_+(G;S) \geq \dist (S\to \overline S)$. Thus 
\[\throtplus (G) \geq \min \{|S| + \dist (S\to \overline S ): S \text{ is a PSD zero forcing set} \}.\] 
\end{rem}

\begin{ex}\label{ex:thm25notbest}  Consider the graph $P_{10}\cp P_2$ shown in Figure \ref{fig:thm25notbest} with the set $S_0$ of four blue vertices.  Observe that $\Z_+(P_{10}\cp P_2)=2$ and for any positive semidefinite zero forcing set $S$ of cardinality two, $\dist (S\to \overline S)\ge 5$.  Thus, $\ptp(P_{10}\cp P_2;S)\ge 5$ and $\throtplus(P_{10}\cp P_2;S)\ge 7$ for such $S$.  However, $\ptp(P_{10}\cp P_2;S_0)=2$ and  $\throtplus(P_{10}\cp P_2;S_0)=6$. 

\begin{figure}[h!]
 \begin{center}
 \scalebox{.6}{\includegraphics{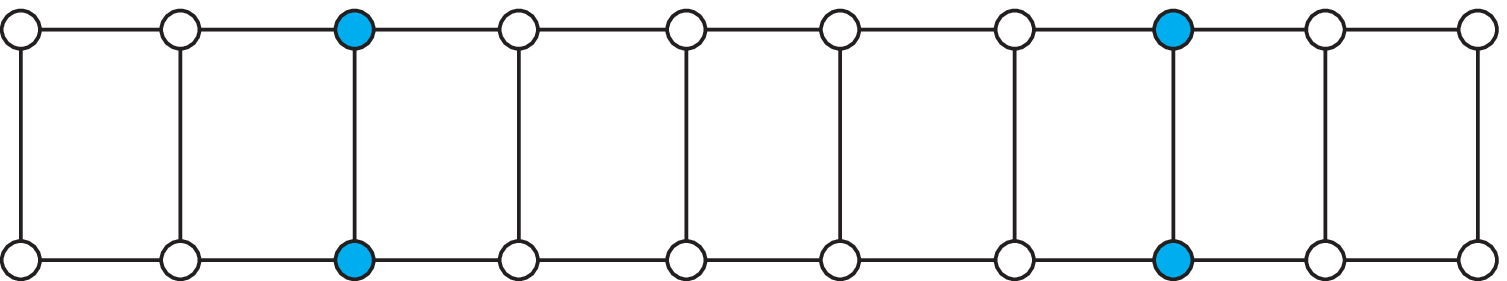}} 

\caption{The graph $P_{10}\cp P_2$ with blue vertices in set $S_0$ that realizes $\throtplus(P_{10}\cp P_2)=6$ \label{fig:thm25notbest}
}\vspace{-20pt}
\end{center}
\end{figure}
\end{ex}

Next we give some easy upper bounds. 

\begin{prop}\label{indepBnd}
Suppose $G$ is a connected graph on $n$ vertices. 
Then \[\thp{G}\leq n -\alpha(G) +1. \]
\end{prop}

\begin{proof}
Let $A\subseteq V(G)$ be a set of independent vertices with $|A|=\alpha(G)$. Let $S=\OL A$. Then each component of $G[A]=G-S$ is an isolated vertex, so $\ptpm{G}{S}=1$. Thus, \[\thp{G}\leq |S| + \ptpm{G}{S} = (n-\alpha(G))+1. \qedhere \]
\end{proof}
The family $\G$ (see Definition  \ref{def:FamilyG}) of graphs $G$ that have $\alpha(G)=2$ and  $\throtplus(G)=|V(G)|-1$ shows that the bound in Proposition \ref{indepBnd} is tight.

\begin{prop}\label{noinduce4hi} Let $G$ be a graph of order $n$ and $H$ be an induced subgraph of $G$ of  order $k$.  Then $\throtplus(G)\leq n-k+\throtplus(H)$.
\end{prop}
\bpf  Choose a set $S\subseteq V(H)$ such that $\throtplus(H;S)=\throtplus(H)$.  Define $S'=(V(G)\setminus V(H))\cup S$. Then  $\throtplus(G;S')\leq n-k+\throtplus(H)$.
\epf

In \cite[Theorem 1]{BY13throttling}, Butler and Young established an upper bound on standard throttling number in terms of the standard zero forcing number $\Z(G)=k$ and the number of vertices $n$: $\throt(G)\leq (2 k +1) \lc \sqrt{n} \rc +k$. Since $\throtplus(G)\leq\throt(G)$, this gives an upper bound for $\throtplus(G)$.  This bound could be improved slightly for paths and cycles ($\Delta(G)=2$), but we determine the positive semidefinite throttling numbers of these graphs in the next section, so we do not pursue modification of this upper bound further.

\section{Paths and cycles}\label{pathsAndCycles} 

In this section we provide  constructions to show that the bound in Proposition \ref{d2bound} for $\Delta=2$ is tight and is attained by  paths and cycles.  We begin with a preliminary lemma.

\begin{lem} 
Let $n \geq 1$. Define $k$ to be the largest even natural number such that $\frac{k^2}{2} \leq n$, and $r=n-\frac{k^2}{2}$. Then \[ \left \lceil \sqrt{2n}-\frac{1}{2} \right \rceil = \left\{\begin{array}{ll} k & \text{if } 0 \leq r < \frac{k}{2} + 1\\[1em]
k+1 & \text{if } \frac{k}{2} + 1 \leq r < \frac{3}{2}k + 2 \\[1em] k+2 & \text{if } \frac{3}{2}k + 2 \leq r < 2k + 2 \end{array}\right.
\]

\label{lemDelta2}
\end{lem}

\bpf 
For every $n \geq 1$, note that $\left \lceil \sqrt{2n} - \frac{1}{2} \right \rceil = \left \lceil \sqrt{2(\frac{k^2}{2}+r)} - \frac{1}{2} \right \rceil= \left\lceil \sqrt{k^2+2r}-\frac{1}{2}\right \rceil$. For $r=0$, $\left \lceil \sqrt{2n} - \frac{1}{2} \right \rceil = \left \lceil k - 1/2 \right \rceil= k$. 

Observe that
\bea
 \left \lceil \sqrt{k^2 + 2r} -\frac{1}{2} \right \rceil > k   
 &\Leftrightarrow&  \sqrt{k^2 + 2r} -\frac{1}{2} > k \\
 &\Leftrightarrow& k^2 + 2r > k^2 + k + \frac{1}{4} \\
 &\Leftrightarrow& r > \frac{k}{2} + \frac{1}{8}\\
 &\Leftrightarrow& r \geq \frac{k}{2} + 1 
\eea
since $k$ is even and $r$ is an integer. This implies that $\left \lceil \sqrt{k^2 + 2r} -\frac{1}{2} \right \rceil > k$ if and only if $r\geq \frac{k}{2} + 1$. Therefore,  $\left \lceil \sqrt{k^2 + 2r} -\frac{1}{2} \right \rceil= k$ for $0\leq r< \frac{k}{2}+1$.

By a similar argument $\left \lceil \sqrt{k^2 + 2r} -\frac{1}{2} \right \rceil > k+1$ if and only if $r\geq \frac{3k}{2} + 2$. Therefore,  $\left \lceil \sqrt{k^2 + 2r} -\frac{1}{2} \right \rceil= k+1$ for $\frac{k}{2}+1\leq r< \frac{3k}{2}+2$. 

Finally, another similar argument shows that $\left \lceil \sqrt{k^2 + 2r} -\frac{1}{2} \right \rceil > k+2$ if and only if $r\geq \frac{5k}{2} + 4$.
Therefore,  $\left \lceil \sqrt{k^2 + 2r} -\frac{1}{2} \right \rceil= k+2$ for $\frac{3k}{2}+2\leq r< \frac{5k}{2}+4$. However, notice that if $r \geq 2k+2$, then $n=\frac{k^2}{2}+r\geq\frac{k^2}{2}+\frac{4k}{2} + \frac{4}{2}=\frac{(k+2)^2}{2}$. 
This implies that we did not pick the largest even natural number $k$ such that $\frac{k^2}{2}\le n$. 
Thus,  $\left \lceil \sqrt{k^2 + 2r} -\frac{1}{2} \right \rceil = k+2$ for $\frac{3k}{2}+2\leq r< 2k+2$. \epf

\begin{thm} [PSD throttling for paths] \label{thmThrP}
Let $n\geq 1$. Then \[\throtplus (P_n)= \left \lceil \sqrt{2n}-\frac{1}{2} \right \rceil.
\]
\end{thm} 

\bpf 

By Proposition \ref{d2bound}, it suffices to construct a positive semidefinite zero forcing set $S$ that gives $\throtplus(P_n;S)$ equal to the value in Lemma \ref{lemDelta2}. 

Label the vertices $v_1, v_2, \ldots ,v_n$ in path order. It may help to visualize the path in a snake pattern using rows of length $k$ as in Figure \ref{PathFig}. There are at least $\frac{k}{2}$ rows and $r$ is the number of extra vertices outside the $\frac{k}{2} \times k$ rectangle. 

Let $B =  \{v_{\frac{k^2}{2} - i(k+1)}\!:\! i=0,\dots,\frac{k}{2}-1 \}$; clearly, $|B|=\frac{k}{2}$. In other words, we obtain $B$ by choosing the $(\frac{k^2}{2})^{th}$ vertex and then working backwards, choosing every $(k+1)^{th}$ vertex.  Therefore, there is a subpath of $k$ white vertices between two blue vertices in the rectangle. Thus, there is only one blue vertex in every row.
Note that the last vertex chosen is $v_{\frac{k}{2} + 1}$. If $0 \leq r < \frac{k}{2}$, then each vertex is within distance $\frac{k}{2}$ of a vertex in $B$. So  $|B| + \ptp({P_n};B) \leq \frac{k}{2} + \frac{k}{2} = k$. If $\frac{k}{2} + 1 \leq r < \frac{3}{2}k + 2$, then let $q = \text{min}(n, \frac{k^2}{2} + k + 1)$ and let $S = B \cup \{v_q\}$. Note that since $r < \frac{3}{2}k + 2$, every vertex in $P_n$ is at distance at most $\frac{k}{2}$  from a vertex in $S$. So  $|S| + \ptp(P_n;S) \leq \frac{k}{2} + 1 + \frac{k}{2} = k + 1$. If $\frac{3}{2}k + 2 \leq r < 2k + 2$, then let $S = B \cup \{v_q\} \cup \{v_n\}$. Since $r < 2k + 2$, every vertex in $P_n$ is at distance at most  $\frac{k}{2}$ from a vertex in $S$. So in this case   $|S| + \ptp(P_n;S) \leq \frac{k}{2} + 2 + \frac{k}{2} = k + 2$. 
\epf 

\vspace{-10pt}
\begin{figure}[h] 
\begin{center}
\scalebox{.8}{\includegraphics{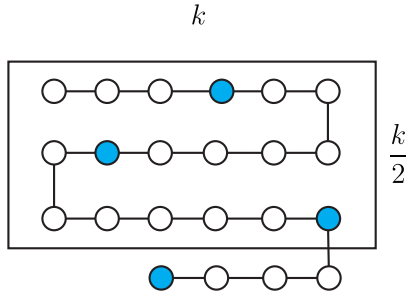}}\\
\caption{The construction for $P_{22}$ with $k=6$ and $r=4$. }\vspace{-10pt}
\label{PathFig}
\end{center}
\end{figure}

\begin{thm} [PSD throttling for cycles] \label{thmThrC}

For $n \geq 4$, $\throtplus(C_n)= \lc \sqrt{2n}-\frac{1}{2} \rc$. 

\end{thm}

\begin{proof}
For each value given in Lemma \ref{lemDelta2}, we will construct a positive semidefinite zero forcing set $S$ that gives $\throtplus(C_n;S)$ equal to the value. By Proposition \ref{d2bound}, this is sufficient. 

Label the vertices of the cycle $v_1, \ldots, v_n$ in cycle order.  
Let $S= \left\{v_{1 + i(k+1)}\!:\! i=0, \ldots, \lf\frac{n-1}{k+1}\rf \right\}$ where $k$ is the largest even integer such that $n \geq \frac{k^2}{2}$. So $S$ is obtained by choosing $v_1$ and then every $k+1$ vertex around the cycle. The cycle is cut into paths because  $n\ge 4$ implies  $|S|=\lf\frac{n-1}{k+1}\rf+1\ge 2$. Every vertex is within a distance of $\frac{k}{2}$ of a vertex in $S$,  so $\ptp(C_n;S)= \frac{k}{2}$.

Let us visualize the cycle snaking in a rectangle that is $k+1$ vertices long as in Figure \ref{fig:Cycle}. 
So there is exactly one colored vertex in every row, alternating between 
the leftmost and rightmost columns. 
Thus,  determining the number of vertices in $S$  is the same as determining the number of rows needed to enclose all of our vertices in this snake pattern. 

Suppose our rectangle has $ \frac{k}{2}$ rows.  Then it encloses up to $\frac{k}{2}(k+1)=\frac{k^2}{2}+\frac{k}{2}$ vertices. If $0 \leq r \leq \frac{k}{2}$, then there are $\frac{k}{2}$ rows. Therefore, $|S|+\ptp(C_n;S) \leq \frac{k}{2}+\frac{k}{2}=k$. 

Now suppose our rectangle has $\frac{k}{2}+1$ rows.  Then it encloses up to $(\frac{k}{2}+1)(k+1)=\frac{k^2}{2}+\frac{3k}{2}+1$ vertices. So, if $\frac{k}{2}+1 \leq r \leq \frac{3k}{2}+1$, then there are $\frac{k}{2}+1$ rows. Thus, $|S|+\ptp(C_n;S) \leq \frac{k}{2}+1+\frac{k}{2}=k+1$.

Finally, suppose our rectangle has $\frac{k}{2}+2$ rows.  Then  it encloses up to $(\frac{k}{2}+2)(k+1)=\frac{k^2}{2}+\frac{5k}{2}+2$ vertices. So, if $\frac{3k}{2}+2 \leq r < 2k+2$, then there are $\frac{k}{2}+2$ rows. Thus, $|S|+p(C_n;S) \leq \frac{k}{2}+2+\frac{k}{2}=k+2$.
\end{proof}

\vspace{-6pt}
\begin{figure}[h]
\begin{center}
\scalebox{.8}{\includegraphics{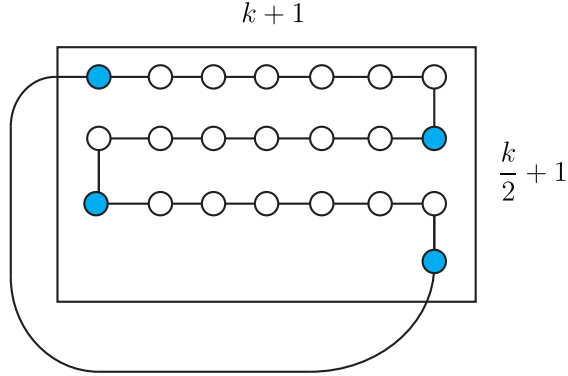}}\\
\caption{The construction for $C_{22}$ with $k=6$ and $r=4$. }\vspace{-10pt}
\label{fig:Cycle}\vspace{-10pt}
\end{center}
\end{figure}

Observe that $C_3=K_3$ and $\throtplus(C_3)=3>2=\lc\sqrt{2n}-\frac 1 2\rc$, so this case is not covered by Theorem \ref{thmThrC}.

\section{Trees}\label{trees} 

In this section we explore the behavior of positive semidefinite zero forcing and positive semidefinite throttling on trees. In particular, we demonstrate that positive semidefinite throttling is subtree monotone for trees and exhibit a family of trees that shows that the bound in  Theorem \ref{thmDeltaGreater2} is tight.

\begin{lem} \label{noLeaf}
Let $T$ be a tree on $n\geq 3$ vertices. Then there exists a set $S\subseteq V(T)$  that  contains no leaves  such that
$\thp{T}=\throtplus(T;S).  $
\end{lem}

\begin{proof}
Let $S'$ be a set of vertices  such that $\thp{T}=\left|S' \right| + \ptpm{T}{S'}$. Suppose $v\in S'$ is a leaf. Then $N(v)=\left\{u\right\}$ for some $u\in V(T)$. 
If $u\in S'$, then $S=S'\setminus\{v\}$ must also realize $\thp{T}$, as it has one fewer vertex and $\ptpm{T}{S}\leq \ptpm{T}{S'}+1$.
Suppose $u\not\in S'$.  Define $S=(S'\setminus\{v\})\cup\{u\}$ and note that $\ptpm{T}{S'}\geq 1$.  Since $\left|S\right| = \left|S'\right|$ and $v$ is colored blue in the first time step, $S$ must also realize $\thp{T}$.  Repeat this leaf removal process as needed.
\end{proof}

\begin{lem}\label{trim}
Suppose $T$ is a tree of order $n\geq 3$, and $v\in V(T)$ is a leaf. Then
\[\thp{T-v}\leq\thp{T}.  \]
\end{lem}

\begin{proof}
By Lemma \ref{noLeaf}, there is a set $S\subseteq V(T)$ that realizes $\thp{T}$ and contains no leaves. Then $S\subseteq V(T-v)$, and $\ptpm{T-v}{S}\leq \ptpm{T}{S}$. Thus, $\thp{T-v}\leq\thp{T}. $
\end{proof}

\begin{thm}\label{subtree}
If $T$ is a tree with subtree $T'$, then
\[\thp{T'}\leq\thp{T}.  \]
That is, positive semidefinite throttling is subtree monotone for trees.
\end{thm}

\begin{proof}
If $|V(T')|=1$, $\throtplus(T')=1\leq \throtplus(T)$. Suppose $|V(T')|\geq 2$ and $|V(T)|\geq 3$. By Lemma \ref{trim}, removing a leaf from $T$ cannot increase the positive semidefinite throttling number. Since any subtree $T'\leq T$ can be attained by repeated removal of leaves from $T$, $\thp{T'}\leq \thp{T}$.
\end{proof}

Although positive semidefinite throttling is subtree monotone for trees, it is not forest monotone (disconnected subgraph monotone) for trees.  For example, consider the path $P_{22}$.  As proved in Theorem \ref{thmThrP} and  illustrated in Figure \ref{PathFig}, $\throtplus(P_{22})=7$.  But the induced subgraph $H$ consisting of every other vertex has $\throtplus(H)=11$ because $H$ consists of 11 isolated vertices. 
Positive semidefinite throttling is also not connected induced subgraph monotone for graphs that are not trees as seen in the next example.\vspace{-5pt}

\begin{figure}[h!]
\begin{center}
\scalebox{.6}{\includegraphics{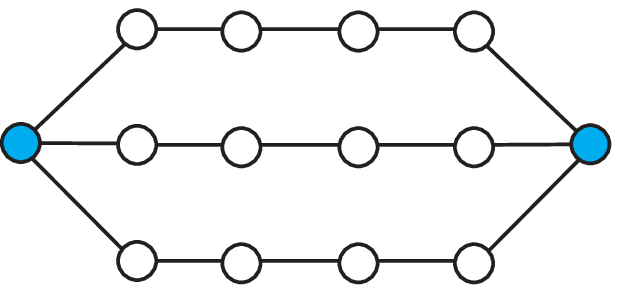}} \hspace{30pt}
\scalebox{.6}{\includegraphics{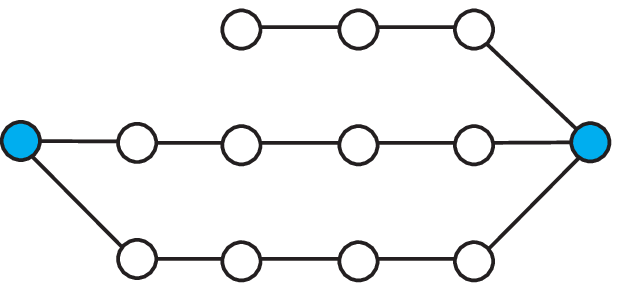}}\\
$G$ \hspace{1.8 in}$G'$
\caption{Graph $G$ with $\throtplus(G)=4$ and subgraph $G'$ with $\throtplus(G')=5$.}
\label{fig:SubMono}\vspace{-15pt}
\end{center}
\end{figure}

\begin{ex}
Let $G$ and $G'$ be the graphs in Figure \ref{fig:SubMono}. The set $\hat S$ of blue vertices has $\pt_+(G; \hat S)=2$, so $\throtplus(G) \leq 4$. Since $G$ is not a tree, $Z_+(G) \geq 2$. Since $\dist(S\to \overline S) \geq 2$ for all minimum positive semidefinite zero forcing sets $S$, it follows that $\throtplus(G) \geq 4$. 

The set $\hat S'$ of blue vertices of $G'$ in Figure \ref{fig:SubMono} has $\pt_+(G'; \hat S')=3$ so $\throtplus (G') \leq 5$. Again, since $G'$ is not a tree, it follows that $Z_+(G) \geq 2$. For all positive semidefinite zero forcing sets $S$, $|S|=2$ implies $\dist(S\to  \overline S) \geq 3$, and  $|S|=3$ implies  $\dist(S\to \overline{S}) \geq 2$.  Thus, $\throtplus(G') \geq 5$. 
\end{ex}

\begin{cor}\label{treeBound1}
If $T$ is a tree with diameter $d$, then
\[\ceil{\sqrt{2(d+1)}-\frac{1}{2}}\leq\thp{T} \leq \ceil{\frac{d}{2} }+1.\]
\end{cor}

\begin{proof}
Since $T$ has diameter $d$, it contains a path of length $d$. Thus, $P_{d+1}$ is a subtree of $T$, and the lower bound follows from Theorem \ref{subtree} and Proposition \ref{d2bound}.
In \cite{PSDpropTime}, Warnberg shows that for any tree, $\ptp(T)=\ceil{\frac{d}{2}}$, which establishes the upper bound. 
\end{proof}

A \emph{full binary tree of height $h$} is a tree with {\em root} vertex $r\in V(T)$ such that
\begin{enumerate}
\item $\deg(r)=2$,
\item For all $ v\in V(T)$, $\dist (v,r) \leq h$,
\item For all $ v\in V(T)$ with $v\neq r$, $\displaystyle{ \operatorname{deg}(v)=\begin{cases} 3 & \dist(v, r) <h \\ 1 &\dist (v,r) =h. \end{cases}  }$
\end{enumerate}
A vertex $u$ is called a \emph{parent} of vertex $v$ if $v\sim u$ and $\dist (u, r)=\dist(v, r)-1$. In this case, $v$ can be referred to as a \emph{child} vertex of $u$. Further, $v$ is a \emph{sibling} of $w$ if $w$ and $v$ share a parent vertex.

\begin{prop}\label{binary}
If $T$ is a full binary tree of height $h$, then $\throtplus(T)=h+1$.
\end{prop}

\begin{proof}
Let $r\in V(T)$ be the root of $T$. Note that by Corollary \ref{treeBound1}, $\throtplus(T)\leq h+1$, and that this bound is realized by choosing $S=\{r\}$. Further, for all vertices $v\neq r$, there exists $v'\in V(T)$ such that $\dist(v, v') >h$, and  $\throtplus(T,\{v\})>h+1$.

Let $S\subseteq V(T)$ be a choice of starting vertices such that $\throtplus(T;S)=\throtplus(T)$, and suppose $|S|\geq 2$.    We show there is a set $S'$ such that $|S'|<|S|$ and $\throtplus(G,S')\le \throtplus(G,S)$, which 
allows us to reduce to the previous case where $S$ is a single vertex. Let $v\in S$ be such that $\dist(v,r)\ge \dist(v',r)$ for all $v'\in S$, let $u$ be the parent vertex of $v$, and let $w$ be the sibling vertex of $v$. Consider the following cases.

\noindent \textbf{Case 1:}
If $u\in S$, define $S'=S\setminus \{v\}$.  Then $\throtplus(T;S')\leq\throtplus(T;S)$
because you reduce the size of $S$ by one and increase propagation time by at most one.

\noindent\textbf{Case 2:}
If $u \notin S$, but $w$ in $S$, define $S'=(S \setminus \{v,w\})\cup\{u\}$.  Then $\throtplus(T;S')\leq\throtplus(T;S)$ 
for the same reason as in Case 1.

\noindent\textbf{Case 3:}
If $u,w\notin S$, then $\throtplus(T;(S\setminus\{v\})\cup\{u\})\leq \throtplus(T;S).$ Since $v$ was of maximum distance from the root in $S$, forcing would have to travel through the parent $u$ to $w$ (and possibly beyond). Locally, taking the parent vertex then reduces the propagation time by one. Globally, propagation time either does not change, or is reduced by one.  Applying this strategy to all vertices at maximum distance would reduce the positive semidefinite throttling time, contradicting the minimality of $\throtplus(T;S)$.  So in this process we must encounter one of Case 1 or Case 2.
\end{proof}

\begin{rem}\label{binfam}
Note that Proposition \ref{binary}, in conjunction  with Theorem \ref{subtree}, determines the throttling number for a large family of trees. Specifically, for any tree $T$ that has a full binary subtree of height $h$ and diameter $2h$, $\throtplus(T)=h+1$.  
\end{rem}

We will now construct a family of graphs for which the bound given in Theorem \ref{thmDeltaGreater2} is tight. Define the tree $T(\Delta,h)$ as follows: Every vertex in the tree has degree $\Delta$ or one, and all leaves are distance $h$ from the (unique) center of the tree. Essentially, this is a full $(\Delta-1)$-ary tree of height $h$, where the root (center) has an extra branch. See Figure \ref{fig:DeltaTree}.

\begin{figure}[ht] 
\begin{center}
\scalebox{.7}{\includegraphics{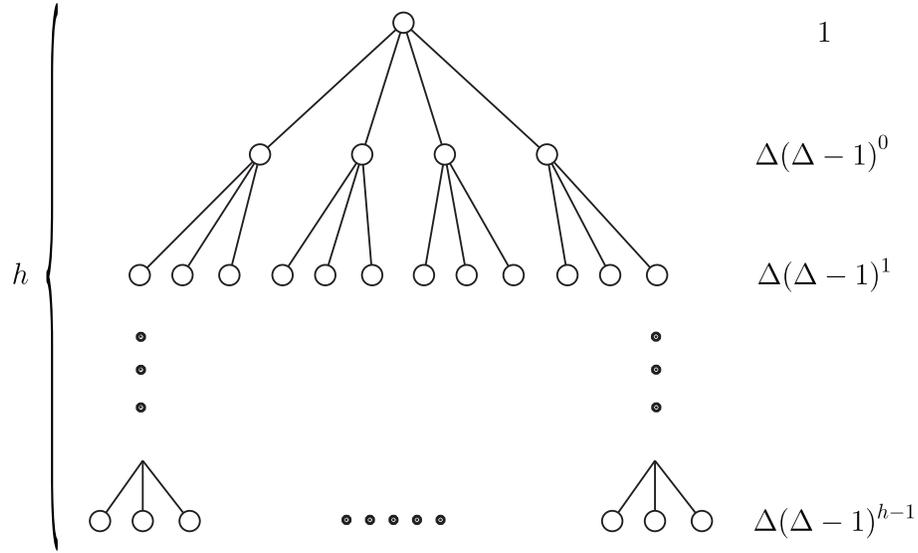}}
\caption{The tree $T(\Delta,h)$ \label{fig:DeltaTree}}\vspace{-10pt}
\end{center}
\end{figure}

\begin{prop}\label{TDHBound}
For all $\Delta\geq 3$ and $h\geq 2$, $\thp{T(\Delta,h)}=h+1$. Furthermore, the bound given in  Theorem \ref{thmDeltaGreater2}
is tight for $T(\Delta,h)$.
\end{prop}

\begin{proof}
Note that $T(\Delta,h)$ has a full binary tree of height $h$ as a subtree, and has diameter $2h$. As per Remark \ref{binfam}, $\throtplus(T(\Delta,h))=h+1$. Furthermore, 
\[n=|V(T(\Delta,h))|=1+\Delta\sum_{k=0}^{h-1}(\Delta-1)^k=\frac{\Delta(\Delta-1)^h-2}{\Delta-2}.  \]
Thus, Theorem \ref{thmDeltaGreater2} gives the bound
\[\throtplus(T(\Delta,h))\geq 1+ \log_{\Delta-1}\left(\frac{\left( \Delta-2\right)\left(\frac{\Delta(\Delta-1)^h-2}{\Delta-2}\right)+2}{\Delta} \right)=h+1. \qedhere  \]
\end{proof}\vspace{-10pt}


\section{Extreme positive semidefinite throttling}\label{extreme} 
In this section, we investigate graphs with extreme positive semidefinite throttling numbers. Specifically, we classify graphs with $\throtplus(G)\in \{1, 2, 3\}$ and connected graphs with $\throtplus(G)\in \{n-1,n\}$.
\subsection{Low positive semidefinite throttling number}

\begin{obs}
$\throtplus(G)=1$ if and only if $G$ is a single vertex.
\end{obs}
\begin{prop} \label{th+2}
For a graph $G$ of order at least $2$, $\throtplus (G) = 2$ if and only if $G=K_{1,n-1}$ or $G=2K_1$. 
\end{prop}
\bpf
If $G=2K_1$, then $Z_+(G)=2$ and $\ptp(G;V(G))=0$. If $G=K_{1,n-1}$, then choose the center vertex of $G$ to be $S$ and $\ptp(G;S)=1$.
Now suppose $\throtplus (G; S)=2$ for some $S$. Then $|S| \in \{1,2\}$. If $|S|=1$, then the propagation time is one. This means $v\notin S$ implies $ v\in N(S)$, and $v$ is its own component in $G-S$. So $G=K_{1,n-1}$. If $|S|=2$, then propagation time is zero, so $|V(G)|=2$. 
\epf

\begin{rem}
We note that Proposition \ref{th+2} is in stark contrast to $\throt(K_{1,n-1})$: For standard zero forcing, $K_{1,n-1}$ requires that at least $n-2$ of the leaves be in the zero forcing set. 
After coloring these leaves, the remaining two vertices can be forced in two steps,  another vertex can be colored blue initially to obtain  propagation time one, or all vertices can be colored blue initially. All  of these zero forcing sets result in a standard throttling number of $n$.
\end{rem}

\begin{thm} \label{th+3}
For a graph $G$, $\throtplus (G)=3$ if and only if
\begin{enumerate}[(1)]
\item\label{th3c1} $G$ is disconnected and 
\ben[a)]
\item $G$ is $3K_1$, or 
\item $G$ has two components,  each component is a copy of $K_{1,n-1}$ or $K_1$, and at least one component has order greater than one.   
\een
\item\label{th3c2} $G$ is a tree with diameter three or four, or

\item\label{th3c3} $G$  is connected and there exist $v,u \in V(G)$ such that: 
\begin{enumerate}[a)]
\item $G$ has a cycle, or $G$ is a tree with $\diam G=5$, 
\item $N(u) \cup N(v)=V(G)$,
\item $\deg(w) \leq 2$ for all $w \not\in  \{v,u\}$, and
\item if $w_1,w_2 \in N(u)$ or $w_1,w_2 \in N(v)$, then $w_1$ is not adjacent to $w_2$. 
\end{enumerate} 
\end{enumerate}
\end{thm}

\begin{proof}
First, we show that every graph $G$ satisfying one of conditions \eqref{th3c1} -- \eqref{th3c3} has $\throtplus(G)=3$.   Clearly $\throtplus(3K_1)=3$, and the two-component graphs in \eqref{th3c1} have positive semidefinite propagation time equal to one. 
Let $G$ be a tree of diameter three or four and $v$ be a center vertex, 
so every vertex in the tree is a distance at most two away from $v$. 
Thus, $\pt_+(G;\{v\})\le 2$ and $\throtplus(G) \leq 3$. Since $G$ has diameter three or four, $G \neq K_{1,n-1}$. Therefore,  $\throtplus(G)  \neq 2$ by Proposition \ref{th+2}. Thus, $\throtplus(G)  = 3$. 

Suppose $G$ is a graph satisfying \eqref{th3c3} 
and let $S$ be the set of the designated vertices  $v$ and $u$. The induced subgraph  $G-S$ has isolated vertices and copies of $P_2$ as components. In each $P_2$ component, one end is adjacent to $u$ and the other  is adjacent to $v$, because $N(u) \cup N(v)=V(G)$. Thus, $\ptp(G;S)=1$ and $\throtplus(G;S) \leq 3$. 
Since $G$ has a cycle or  $\diam G=5$,  $\throtplus(G;S) \neq 2$  by Proposition \ref{th+2}, so $\throtplus(G; S)  = 3$.

 Now suppose $G$ is a graph such that $\throtplus(G)=3$. Let  $n$ denote the order of $G$, $d=\diam G$, and $S\subseteq V(G)$ such that $\throtplus(G;S)=3$.  If $G$ is disconnected and has three components, then $\Z_+(G)= 3$, $\ptp(G)=0$, and $G=3K_1$.  
 If $G$ has two components, then $\Z_+(G)= 2=|S|$ and $\ptp(G)=1$, so each component is $K_1$ or $K_{1,n-1}$ (with at least one $K_{1,n-1}$), or $|S|=3$; the latter was already discussed.  
 Thus  $G$ is disconnected implies $G$ is of the form  \eqref{th3c1}.  So we assume $G$ is connected.  If $|S|=3$, then $G=K_3$, which is of the form \eqref{th3c3}.  So we assume $n\ge 4$ and $|S|\le 2$.
   If $G$ is a tree, then  $3\le \diam G\le 5$ by 
   Corollary \ref{treeBound1}.  The cases of diameter three and four are covered by \eqref{th3c2}.  
   
 Now we have the following information: $G$ is connected,  $n\ge 4$,  $|S|\le 2$, and  $G$ is not a tree of  diameter less than five.   This implies $|S|= 2$, because $\Z_+(G) \ge 2$ for any graph that is not a tree, and if $G$ is a tree with $d=5$, then any one vertex is at a distance at least three from some vertex. 
 Let $S=\{v,u\}$.  Then all other vertices are adjacent to at least one of $v,u$ (and there is a vertex $w\ne u,v$). Every component of  the subgraph $G-S$ is an isolated vertex or a copy of $P_2$ with one end  adjacent to $u$ and the other adjacent to $v$, since both are forced in the first time step. Therefore,  $\deg(w) \leq 2$ for all $w \in \OL S$. 
  If $w_1, w_2$ are both adjacent to $v$, then they cannot be adjacent in order to have propagation time equal to one; the case $w_1, w_2$  both adjacent to $u$ is similar. Thus, $G$ is of the form \eqref{th3c3}. 
\end{proof}

\subsection{High positive semidefinite throttling number}

In this section, we classify the families of connected graphs of order $n$ with positive semidefinite throttling number equal to $n$ or $n-1$.

\begin{prop}\label{th+n}
Let $G$ be a connected graph. Then
$\throtplus(G)=n$ if and only if $G=K_n$.
\end{prop}

\begin{proof}
Assume $\throtplus(G)=n$. Then by Proposition \ref{indepBnd}, $\alpha(G)=1$ and $G$ is complete. 
Any positive semidefinite zero forcing set of $K_n$ contains at least $n-1$ vertices, so $\throtplus(K_n)=n$.\end{proof}

\begin{rem}
Since $n=\throtplus(G)\le \throt(K_n)\le n$,  $\throt(K_n)=n$ (this is also easy to see directly). 
However, there are other graphs $G$ of order $n$, such as $K_{1,n-1}$, that have  $\throt(G)=n$.  
\end{rem}

In order to classify the connected graphs of order $n$ with positive semidefinite throttling number equal to $n-1$, we define a family $\G$ of graphs.

\begin{defn}\label{def:FamilyG}The set $\mathcal{G}$ consists of all graphs $G$ such that 
 $\alpha(G)=2$ and
 $G$ does not have an induced $C_5$, house, or double diamond  subgraph (see  Figure \ref{fig:highthrotforbid}).\vspace{-10pt}
\begin{figure}[ht] 
\begin{center}
\scalebox{.5}{\includegraphics{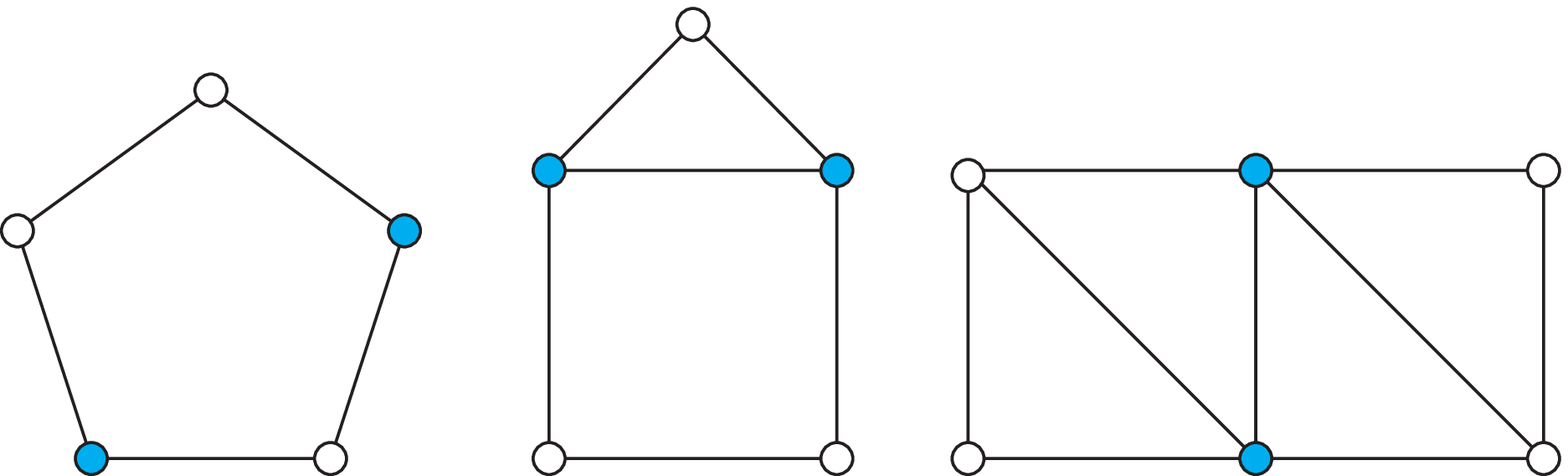}}\\
$\null\qquad C_5\ \ \qquad\qquad$ \ house$\qquad\ \ \ $ double diamond
\caption{Graphs forbidden as induced subgraphs of $\G$.\label{fig:highthrotforbid}}\vspace{-15pt}
\end{center}
\end{figure}
\end{defn}

The next proposition shows that every connected graph of order $n$ with positive semidefinite throttling number equal to $n-1$ is in $\mathcal G$.

\begin{prop}\label{th+n-1Suf}
Let $G$ be a connected graph on $n$ vertices. If $G\notin\mathcal{G}$,
then $\throtplus(G)\neq n-1$.
\end{prop}

\begin{proof}
First, we show $\alpha(G)\ne 2$ implies $\throtplus(G)\neq n-1$. If $\alpha(G)=1$, then $G$ is complete, so  $\throtplus(G)=n$ by Proposition \ref{th+n}.
If $\alpha(G)>2$, then $\throtplus(G)\leq n-\alpha(G)+1\leq n-2$ by Proposition \ref{indepBnd}. 

Positive semidefinite zero forcing sets $S$ that attain $\throtplus(H;S)= |V(H)|-2$ where $H$ is  one of $C_5$, the house, or the double diamond are shown as blue vertices in Figure \ref{fig:highthrotforbid}.  Then by Proposition   \ref{noinduce4hi},  $\throtplus(G)\leq n-|V(H)| +(|V(H)|-2)= n-2$. 
%
%
\end{proof}

The next three lemmas will be used together to show that any connected order $n$ graph in $\mathcal G$ has positive semidefinite throttling number equal to $n-1$. 

\begin{lem}\label{KUforcetime2}
Suppose $G\in \mathcal{G}$ and $W$ is a set of white vertices such that $G[W]$ consists of two disjoint cliques, at least  one of which is trivial.  Then 
$\ptpm{G}{\overline W}\ge |W|-1$.  
\end{lem}

\begin{proof} Let $k=|W|$, $S={\overline W}$,  
and  $v$ be the vertex of  a trivial clique of $G[W]$. 
For the sake of contradiction, assume that $W':=W\setminus \{v\}$ is forced in less than $k-1$ time steps by $S$. This implies that two vertices of $W'$ are forced simultaneously; that is,  there exist vertices $u_1,u_2\in W'$ such that $u_1,u_2\in S^{(t)}$.   Therefore, there exist $u'_1, u'_2\in S_{t-1}$ such that $u'_i\to u_i$ at time $t$.
Thus, $u'_iu_j\not\in E(G)$ for $i\ne j$. 
Since $\alpha(G)=2$, we also have that $u'_iv\in E(G)$ for $i=1,2$. There are two cases: either $u'_1u'_2\in E(G)$, or $u'_1u'_2\notin E(G)$. The first case implies that there exists an induced house subgraph, and the second case implies there exists an induced $C_5$ (with $u_1,u'_1,v,u'_2,u_2$ clockwise around  the 5-cycle from lower left in Figure \ref{fig:highthrotforbid} for both graphs). Either case is a contradiction, and therefore at least $k-1$ steps are required to force $W'.$  \end{proof}

\begin{lem}\label{KUforcetime3}
Suppose $G\in \mathcal{G}$ and $W$ is a set of white vertices such that $W$ induces two nontrivial disjoint cliques.  Then $\overline W$ will not force $W.$
\end{lem}

\begin{proof} For the sake of contradiction, assume that $\overline W$ forces $W.$ Let $W_1$ and $W_2$ denote the vertices of the two cliques induced by $W.$ Let $x,y\in W_1$ and $w,z\in W_2$. Since $\overline W$ forces $W,$ there exists $x'$ such that $x'$ will force $x$ in the first time step. This implies that $x' y\notin E(G)$, and so $x'z, x'w\in E(G)$ because $\alpha(G)=2$.  Similarly, there exists $z'\in \overline W$ such that $z'$ will force $z$ in the first time step and $z'w\notin E(G)$, $z' x,z'y\in E(G)$. Notice that $x'x z'z$ is a $C_4$ subgraph. Now there are two cases: either $z'x'\in E(G)$, or $x' z'\notin E(G)$.  In the first case, $G[\{x,x',w,y,z',z\}]$ is  a double diamond with $x,x',w$ in the top row left to right and $y,z',z$ in the bottom row as shown  in Figure \ref{fig:highthrotforbid}. In the second case, $G[\{x', x,y,z,z'\}]$ is a house  with $x', x,y,z,z'$ clockwise around  the 5-cycle from lower left in Figure \ref{fig:highthrotforbid}. In both cases $G\notin \mathcal G$. 
\end{proof}

\begin{lem}\label{KUforcetimeGen}
Suppose $G\in \mathcal{G}$ and $W$ is a set of white vertices such that $|W|\geq 3$ and $G[W]$ is connected. Then 
$\ptpm{G}{\overline W}\ge |W|-1$. 
\end{lem}
\begin{proof}  Let $W=\{u_1,\dots,u_{k}\}$ with $k\ge 3$. 
First we show that  it is not possible to have two vertices forced at the first time step.  For the sake of contradiction, assume first that $u_1,u_2$ are simultaneously forced in the first time step. Then there exist blue vertices $u'_1,u'_2$ such that $u'_iu_j\in E(G)$ if and only if $i=j$. 
Notice that $\{u_3,u'_1\}$, $\{u_1,u'_2\}$, $\{u_2,u'_1\}$ form independent sets.  Since $\alpha(G)=2$, $u'_1u'_2, u_1u_3,u_2u_3\in E(G)$. Now there are two cases, either $u_1u_2\not\in E(G)$ or $u_1u_2\in E(G)$. In the first case, we have an induced $C_5$, and in the second case we have an induced house subgraph (with $u'_1,u_1,u_3,u_2,u'_2$ clockwise around  the 5-cycle from lower left in Figure \ref{fig:highthrotforbid}). Both cases lead  to a contradiction. 

Now the proof proceeds by induction on $k=|W|$.  The fact that two forces cannot be performed in the first time step implies that for $k=3$ we have $\ptpm{G}{\overline W}\ge 2=|W|-1$.
We assume that  $\ptpm{G}{\overline W}\ge |W|-1$ for $|W|\leq k-1$, and show $\ptpm{G}{\overline W}\ge |W|-1$ for $|W|= k$. Exactly one vertex  is forced in the first time step; without loss of generality assume the vertex forced first is $u_1$. Let  $u'_1$ be the blue vertex  that forces $u_1$.  In order to apply the induction hypothesis, we must show that $W\setminus \{u_1\}$ induces a connected subgraph. For the sake of contradiction, suppose that $W\setminus\{u_1\}$ induces at least two components. This implies that there exist vertices $u_2,u_3\in W\setminus\{u_1\}$, each in a different component. Notice that $u'_1$ is not adjacent to $u_2,u_3$ by positive semidefinite zero forcing rules. Therefore, $\{u'_1,u_2,u_3\}$ is an independent set. This is a contradiction. Therefore, $W\setminus\{u_1\}$ induces a connected component of white vertices. By the induction hypothesis, it will take at least $|W|-2$ time steps to force $W\setminus\{u_1\}$. Therefore, it takes at least $|W|-1$ time steps to force $W$.
\end{proof}

\begin{thm}\label{th+n-1}
Let $G$ be a connected graph on $n$ vertices. Then $G\in\mathcal{G}$ if and only if $\throtplus(G)\geq n-1$.
\end{thm}

\begin{proof}
Assume  $G\in\mathcal{G}$ and $S$ is a positive semidefinite zero forcing set.   If $G-S$ is disconnected, then there are at most two components and each is a clique because $\alpha(G)=2$.  If $G-S$ is disconnected, then  Lemma \ref{KUforcetime3} implies at least one component must be trivial, and Lemma \ref{KUforcetime2} implies $\throtplus(G;S)\ge n-1$. If $G-S$ is connected and $|\OL S|\ge 3$, then Lemma \ref{KUforcetimeGen} implies $\throtplus(G;S)\ge n-1$. Clearly $\throtplus(G;S) \ge n-1$ for  $|\OL S|\le 2$.   
The converse was proved in Proposition \ref{th+n-1Suf}.  \end{proof}

\section{Weighted positive semidefinite throttling}\label{sweight} 

In \cite{BY13throttling}, Butler and Young also considered the effect of minimizing a weighted sum  $a|S| + b\pt(G,S)$ over $S\subseteq V(G)$ for fixed $a,b>0$.  
In this section we discuss the effect of  weighting on some  of our previous results.  
We observe that minimizing $a|S|+b\ptp(G;S)$ is equivalent to considering the fixed multiple $b$ times the weighted sum  $\omega|S|+\ptp(G;S)$ for $\omega:=\frac a b >0$.  Minimizing $\omega|S|+\ptp(G;S)$ determines the same optimizing sets $S$ as minimizing $a|S|+b\ptp(G;S)$, and we minimize this version because  it is notationally more convenient.
For $\omega >0$, define \[\throtplus^{\omega}(G)=\min_{S} \{\omega |S| + \ptp(G;S) \}.\]

The proof of the next result is analogous to the proof of Proposition \ref{d2bound}, and is omitted.  

\begin{prop}\label{d2boundwt}
Suppose $G$ is a graph with $\Delta(G) = 2$.   Then 
 \[ \throtplus^{\omega}(G)\geq \left \lceil \sqrt{2\omega n}-\frac{1}{2} \right \rceil.\]
\end{prop}

\begin{thm}\label{thmDeltaGreater2wt}
Let $G$ be a  graph of order $n$, $\Delta:=\Delta(G)\geq 3$, and $s_0:=\Z_+(G)$.  For $\omega \ge \frac{1}{\ln(\Delta - 1)}\left( \frac{(\Delta -2)n}{s_0(\Delta - 2)n + 2s_0^2} \right)$, 
\[ \throtplus^\omega(G)\ge \lc \omega s_0 + \log_{(\Delta - 1)}\left(\frac{(\Delta-2)n + 2s_0}{\Delta s_0}\right)\rc.\]
The hypothesis $\omega \ge \frac{1}{\ln(\Delta - 1)}\left( \frac{(\Delta -2)n}{s_0(\Delta - 2)n + 2s_0^2} \right)$ is satisfied for  $\omega\ge 1$ and for $\omega \ge \frac{1}{s_0\ln(\Delta - 1)}$. 

If  $\omega < \frac{1}{\ln(\Delta - 1)}\left( \frac{(\Delta -2)n}{s_0(\Delta - 2)n + 2s_0^2} \right)$, then
\[ \throtplus^\omega(G)
= \min\left\{\omega s+ \frac{\ln\left(  \frac{(\Delta - 2)n + 2s}{\Delta s}\right)}{\ln(\Delta - 1)}:s=s_0,s_0+1,\dots, \lc \frac{1}{\omega\ln(\Delta - 1)}\rc\right\}.\] \end{thm}

\bpf  We follow the proof of Theorem \ref{thmDeltaGreater2} and let $s$ and $p$ denote $|S|$ and $\ptp(G;S)$,  respectively. Then the constraint in Lemma \ref{constraint} gives  $\displaystyle p \ge p(s): =\log_{(\Delta - 1)}\left(\frac{(\Delta-2)n + 2s}{\Delta s}\right)$.  Taking the derivative   with respect to $s$ gives $\frac{d}{ds}(\omega s+p(s))=
\omega - \frac{1}{\ln(\Delta - 1)}\left( \frac{(\Delta -2)n}{s(\Delta - 2)n + 2s^2} \right)$.  Since $\omega \ge \frac{1}{\ln(\Delta - 1)}\left( \frac{(\Delta -2)n}{s_0(\Delta - 2)n + 2s_0^2} \right)$ implies $\frac{d}{ds}(\omega s+p(s))\ge 0$, in this case $\omega s+p(s)$ is increasing for $s\ge s_0$, and
\[\throtplus^\omega(G)\ge \lc \omega s_0 + \log_{(\Delta - 1)}\left(\frac{(\Delta-2)n + 2s_0}{\Delta s_0}\right)\rc.\] 

As shown in Theorem \ref{thmDeltaGreater2}, $\omega \ge \frac{1}{\ln(\Delta - 1)}\left( \frac{(\Delta -2)n}{s_0(\Delta - 2)n + 2s_0^2} \right)$ is true for all $\omega\ge 1$.  More generally,
\[ \frac{1}{\ln(\Delta - 1)}\left( \frac{(\Delta -2)n}{s_0(\Delta - 2)n + 2s_0^2} \right) \le  \frac{1}{\ln(\Delta - 1)}\left( \frac{(\Delta -2)n}{s_0(\Delta - 2)n}\right)=\frac{1}{s_0\ln(\Delta - 1)},\] 
so $\omega \ge \frac{1}{\ln(\Delta - 1)}\left( \frac{(\Delta -2)n}{s_0(\Delta - 2)n + 2s_0^2} \right)$ 
is true whenever $\omega \ge \frac{1}{s_0\ln(\Delta - 1)}$.    Finally, when $\omega < \frac{1}{s_0\ln(\Delta - 1)}$, it suffices to minimize over all $s$ between $s_0$ and $\lc\frac{1}{s_0\ln(\Delta - 1)}\rc$.  
\epf

\begin{rem}\label{lowwt} We note that $\omega<1$ favors larger positive semidefinite zero forcing sets, whereas $\omega>1$ favors smaller positive semidefinite zero forcing sets, but we cannot go below $\Z_+(G)$.  Thus, whenever $\omega \ge 1$ and the positive semidefinite throttling number is achieved by a set $S$ with $|S|=\Z_+(G)$, 
\[\throtplus^\omega(G)=\omega \Z_+(G) +\ptp(G).\]
\end{rem}


 Weighted versions of the results for paths and cycles presented in Section \ref{pathsAndCycles} can be obtained from Proposition \ref{d2boundwt} together with modified snaking.


\section{Throttling for standard zero forcing}\label{StndTh} 
In \cite{BY13throttling}, the standard throttling number of a path is constructed by snaking the path  in a rectangle with height $|S|$ and width $\pt(G,S)+1$ and then minimizing throttling subject to the constraint \eqref{stdconst}.
We will show that the standard throttling number for a cycle on $n$ vertices is the same as the standard throttling number for a path on $n$ vertices except in the case $n=m^2$ and $m$ is odd, by adapting the construction for the path. 

\begin{thm}[Standard throttling for cycles] \label{thmStdThrC}
Let $C_n$ be a cycle on $n$ vertices. Define $m$ to be the largest integer such that $m^2 \leq n$ and $n= m^2 +r$. Then \[\throt(C_n)= \begin{cases}
2m-1 & \mbox{if $r = 0$ and $m$ is even} \\
2m & \mbox{if  $0 < r \leq m$ or $(r = 0$ and $m$ is odd$)$} \\
2m+1 & \text{if } m <r <  2m +1
\end{cases}. \]
\end{thm}

\begin{proof}
We begin by snaking a path on $n$ vertices in a rectangle with one endpoint in the upper  left and the snake layout as in Figure \ref{PathFig}, but with different dimensions and coloring; the set $S$  consists of the rightmost vertex in each row (this is the construction in \cite{BY13throttling} reversed vertically to match our figure). 

If no endpoints of the path are in $S$, then adding the edge between the endpoints does not change the propagation time, and we define $S' = S$.  This occurs precisely when there are an even number of rows.  
This is sufficient to show $\throt(C_n)=\throt(P_n)$ in this case because it was shown in \cite{BY13throttling} that $P_n$ realizes the minimum possible throttling number for a graph of order $n$.

So assume $S$ contains exactly one endpoint of the path, and let $s=|S|$ and $p=\pt(P_n,S)$  (so the rectangle has $s$ rows and $p+1$ columns). 
Suppose $s > p+1$ (respectively, $s < p+1$).  We re-snake so that there is one more column and one less row in the rectangle (respectively, one less column and one more row).  Since the new rectangle has room for at least as many vertices, the vertices of $P_n$ fit inside the new rectangle. The new standard zero forcing set  $S'$ consists of the vertices in the right column. 
This re-snake increases (respectively, decreases) the standard propagation time by one and decreases (respectively, increases) the cardinality of the standard zero forcing set by one, so it does not change standard throttling. Now, there are no endpoints in $S'$ because $|S'|$ is even, so we can add an edge without changing the propagation time.  

Since $s$ and $p+1$ can be chosen to differ by at most one, performing an analogous re-snake in the case $s=p+1$ will still  allow the vertices of $P_n$ to fit in the rectangle unless $n=m^2$. 
 If we add the other endpoint of $P_{m^2}$ to $S$ to obtain $S'$, then $\throt(C_{m^2})\le \throt(C_{m^2},S')=2m=\throt(P_{m^2})+1$.  
 The only way to realize $\throt(C_{m^2},S)=2m-1$ subject to \eqref{stdconst} would be to have $|S|=m$, $\pt(G,S)=m-1$, and at each time step   every vertex that turned blue at the previous time step performs a force.  It is not possible to arrange an odd number of vertices on a cycle in such a way that every blue vertex performs a force at the first time step.
\end{proof}

We will now prove some results similar to Lemma $\ref{trim}$ and Theorem $\ref{subtree}$ for regular zero forcing on trees.

\begin{lem}\label{zfTrim}
Let $T$ be a tree and $v\in V(T)$ be a leaf. Then
\[\throt(T-v)\leq \throt(T).  \]
\end{lem}

\begin{proof}
Let  $S$ be a standard zero forcing set of $T$ that realizes $\throt(T)$. 
If $v\notin S$, then $v$ is forced in some time step. 
Let $S'=S$. Then $\pt(T-v,S')\leq \pt(T,S)$. If $v\in S$, consider the vertex $u\in N(v)$. If $u\in S$, then $S'=S\setminus\{v\}$ is a standard zero forcing set of $T-v$ and $\pt(T-v,S')=\pt(T,S)$. If $u\notin S$, then $S'=(S\setminus\{v\})\cup\{u\}$ is a standard zero forcing set of $T-v$ and $\pt(T-v,S')\leq \pt(T,S)$. In all cases, $\left|S'\right|\leq \left|S\right|$ and $\pt(T-v,S')\leq \pt(T,S)$. Thus, $\throt(T-v)\leq \throt(T)$.
\end{proof}

The proof of the next theorem is similar to the proof of Theorem \ref{subtree}, replacing Lemma \ref{trim} by Lemma \ref{zfTrim}, and is omitted.
\begin{thm}
If $T$ is a tree with subtree $T'$, then 
\[\throt(T')\leq \throt(T).  \]
That is, standard throttling is subtree monotone for trees.
\end{thm}



\end{document}